\documentclass[10pt]{amsart}
\usepackage{amsfonts,amssymb,amscd,amsmath,enumerate,verbatim,calc,amsthm}

\headsep=1cm \numberwithin{equation}{section}
\newtheorem{thm}{Theorem}[section]
\newtheorem{cor}[thm]{Corollary}
\newtheorem{lem}[thm]{Lemma}
\newtheorem{prop}[thm]{Proposition}
\newtheorem{defn}[thm]{Definition}

\newcommand{\coker}{\mbox{Coker}\,}
\newcommand{\Hom}{\mbox{Hom}\,}
\newcommand{\Ext}{\mbox{Ext}\,}
\newcommand{\Tor}{\mbox{Tor}\,}
\newcommand{\cx}{\mbox{cx}\,}

\newcommand{\depth}{\mbox{depth}\,}

\newcommand{\Tr}{\mbox{Tr}\,}
\newcommand{\T}{\mbox{T}\,}
\newcommand{\pd}{\mbox{Pd}\,}
\newcommand{\gd}{\mbox{G-dim}\,}
\newcommand{\wgd}{\mbox{w.g.d}\,}
\newcommand{\cid}{\mbox{CI-dim}\,}
\newcommand{\R}{\mathbb{R}}
\newcommand{\N}{\mathbb{N}}

\begin{document}
\bibliographystyle{amsplain}


\title[a note on the depth formula and vanishing of cohomology]
 {a note on the depth formula and vanishing of cohomology}

\bibliographystyle{amsplain}

     \author[A. Sadeghi]{Arash Sadeghi}

\address{Faculty of Mathematical Sciences and Computer,
Tarbiat Moallem University, Tehran, Iran.}

\address{School of Mathematics, Institute for Research in Fundamental Sciences (IPM), P.O. Box: 19395-5746, Tehran, Iran }
\email{sadeghiarash61@gmail.com}

\maketitle
\begin{abstract}
It is proved that if one of the finite modules $M$ and $N$, over a local ring $R$, has reducible complexity and has finite Gorenstein dimension then the depth formula holds, provided   $\Tor^R_i(M,N)=0$ for $i\gg0$. We also study the vanishing of cohomology of a module of finite complete intersection dimension.
\end{abstract}
\section{introduction}
Let $R$ be a local ring. Two $R$--modules $M$ and $N$ satisfy the depth formula if
$$\depth_R(M)+\depth_R(N)=\depth R+\depth_R(M\otimes_RN).$$
The depth formula was first studied by Auslander \cite{A2}. Suppose that $\pd_R(M)<\infty$ and that $q$ is
the largest integer such that $\Tor_q^R(M,N)$ is nonzero. Auslander proved that if or $\depth_R(\Tor_q^R(M,N))\leq1$ either $q=0$, then the formula
\begin{equation}\label{32}
\depth_R(M)+\depth_R(N)=\depth R+\depth_R(\Tor_q^R(M,N))-q
\end{equation}
holds \cite[Theorem 1.2]{A2}.

In \cite{HW}, Huneke and Wiegand showed that two $R$--modules $M$ and $N$ over complete intersection rings satisfy the depth formula provided
$\Tor_i^R(M,N)=0$ for $i>0$. In \cite{I}, Iyengar showed that the depth formula holds for two $R$--modules $M$ and $N$, provided one of the modules has finite complete intersection dimension and $\Tor_i^R(M,N)=0$ for all $i>0$.
In \cite{AY}, Araya and Yoshino  generalized  Auslander's original result. More precisely, they proved that the formula (\ref{32}) holds provided one of the modules has finite complete intersection dimension and $\Tor^R_i(M,N)=0$ for $i\gg0$.
In \cite{BJ}, Bergh and Jorgensen proved that the depth formula holds in certain cases over Cohen-Macaulay rings, provided one of the modules has reducible complexity and $\Tor_i^R(M,N)=0$ for $i>0$.

In this paper, we generalize the Auslander's original result for a module of finite Gorenstein dimension and with reducible complexity.

In section 1, we prove that the formula (\ref{32}) holds provided one of the modules has reducible complexity and has finite Gorenstein dimension
and $\Tor_i^R(M,N)=0$ for $i\gg0$, which is a generalization of \cite[Theorem 2.5]{AY}. Also it can be viewed as a generalization of \cite[Corollary 3.4]{BJ}.

In section 2, we study the vanishing of cohomology of a module of finite complete intersection dimension over a local ring. For an $R$--module $M$ of finite weak Gorenstein dimension and an $R$--module $N$ of finite complete intersection dimension and complexity $c$, it is shown that if there exist an odd number $q\geq1$, and a number $n>\wgd_R(M)$
such that $\Ext^i_R(M,N)=0$ for $i\in\{n,n+q,\cdots,n+cq\}$, then $\Ext^i_R(M,N)=0$ for all $i>\wgd_R(M)$(see Theorem \ref{t1}).
As a consequence, for two $R$--modules $M$ and $N$ of finite complete intersection dimensions, it is shown that if there exist
an odd number $q\geq1$, and $i>\cid_R(M)$ such that $\Ext^j_R(M,N)=0$, for $j\in\{i,i+q,\cdots,i+cq\}$, where $c=\min\{\cx_R(M),\cx_R(N)\}$, then
$\Ext^j_R(M,N)=0$ for all $j>\cid_R(M)$, (see Corollary \ref{c}).
In Theorem \ref{t2}, it is shown that if $\wgd_R(M)<\infty$ and $N$ has reducible complexity such that $\Ext^i_R(M,N)=0$ for $i\gg0$ then
$\wgd_R(M)=\sup\{i\mid\Ext^i_R(M,N)\neq0\}$.

\section{Preliminaries}
Throughout the paper, $R$ is a commutative Noetherian local ring and all modules are finite (i.e. finitely generated)
$R$--modules. Let $$\cdots\rightarrow F_{n+1} \rightarrow F_{n}\rightarrow F_{n-1}\rightarrow\cdots\rightarrow F_0\rightarrow M\rightarrow0$$ be the minimal free resolution of $M$. Recall that the $n^{\text{th}}$ syzygy of an $R$--module $M$ is the cokernel of the $F_{n+1}\rightarrow F_{n}$ and denoted by $\Omega^n_R(M)$, and it is unique up to isomorphism. The $n^{\text{th}}$ Betti number, denoted $\beta_n^R(M)$, is the rank of the free $R$--module $F_n$.
The complexity of $M$ is defined as follows;
$$\cx_R(M)=\inf\{i\in \N \cup 0\mid \exists \gamma\in \R  \text{ such that } \beta_n^R(M)\leq\gamma n^{i-1}\text{ for } n\gg0\}.$$
Note that $\cx_R(M)=\cx_R(\Omega^i_R(M))$ for every $i\geq0$. It follows from the definition that
$\cx_R(M)=0$ if and only if $\pd_R(M)<\infty$.
The complete intersection dimension was introduced by Avramov, Gasharov and Peeva \cite{AGP}. A module of finite complete intersection dimension behaves homologically like a module over a complete intersection.
Recall that a quasi-deformation of $R$ is a diagram $R\rightarrow A\twoheadleftarrow Q$ of local homomorphisms, in which
$R\rightarrow A$ is faithfully flat, and $A\twoheadleftarrow Q$ is surjective with kernel generated by a regular sequence.
The module $M$ has finite complete intersection dimension if there exists such a quasi-deformation for which $\pd_Q(M\otimes_RA)$ is finite.
The complete intersection dimension of $M$, denoted $\cid_R(M)$, is defined as follows;
$$\cid_R(M)=\inf\{\pd_Q(M\otimes_RA)-\pd_Q(A)\mid R\rightarrow A\twoheadleftarrow Q \text{ is a quasi-deformation }\}.$$
By \cite[Theorem 5.3]{AGP}, every module of finite complete intersection dimension has finite complexity.

The concept of modules with reducible complexity was introduced by Bergh \cite{B2}.\\
Let $M$ and $N$ be $R$--modules and consider a homogeneous element $\eta$ in the graded $R$--module
$\Ext^*_R(M,N)=\bigoplus^{\infty}_{i=0}\Ext^i_R(M,N)$. Choose a map $f_{\eta}:\Omega^{|\eta|}_R(M)\rightarrow N$ representing $\eta$, and denote by $K_{\eta}$ the pushout of this map and the inclusion $\Omega^{|\eta|}_R(M)\hookrightarrow F_{|\eta|-1}$.
Therefore we obtain a commutative diagram
$$\begin{CD}
&&&&&&&&\\
  \ \ &&&&  0@>>>\Omega^{|\eta|}_R (M) @>>> F_{|\eta|-1}@>>>\Omega^{|\eta|-1}_R (M) @>>>0&  \\
                                &&&&&& @VV{f_{\eta}}V @VV V @VV{\parallel} V\\
  \ \  &&&& 0@>>> N @>>> K_{\eta} @>>>\Omega^{|\eta|-1}_R (M)@>>>0.&\\
\end{CD}$$\\
with exact rows. Note that the module $K_{\eta}$ is independent, up to isomorphism, of the map $f_{\eta}$ chosen to represent ${\eta}$.
\begin{defn}\emph{
The full subcategory of $R$-modules consisting of the modules having
reducible complexity is defined inductively as follows:}
\begin{itemize}
        \item[(i)]\emph{ Every $R$-module of finite projective dimension has reducible complexity.}
         \item[(ii)]\emph{ An $R$-module $M$ of finite positive complexity has reducible complexity if
                    there exists a homogeneous element $\eta\in\Ext^{*}_R(M,M)$, of positive degree,
                         such that $\cx_R(K_{\eta}) < \cx_R(M)$, $\depth_R(M)=\depth_R(K_{\eta})$ and $K_{\eta}$ has reducible complexity.}
\end{itemize}
\end{defn}
By \cite[Proposition 2.2(i)]{B2}, every module of finite complete intersection dimension has reducible complexity. On the other hand, there are modules having reducible complexity but whose complete intersection dimension is infinite (see for example, \cite[Corollarry 4.7]{BJ}).

The notion of the Gorenstein(or G-) dimension was introduced by Auslander \cite{A1}, and developed by Auslander and Bridger in \cite{AB}.
\begin{defn}
An $R$--module $M$ is said to be of $G$-dimension zero whenever
\begin{itemize}
            \item[(i)]{\emph{the biduality map $M\rightarrow M^{**}$ is an isomorphism;}}
            \item[(ii)]{\emph{$\Ext^i_R(M,R)=0$ for all $i>0$;}}
            \item[(iii)]{\emph{$\Ext^i_R(M^*,R)=0$ for all $i>0$.}}
\end{itemize}
\end{defn}
The Gorenstein dimension of $M$, denoted $\gd_R(M)$, is defined to be the infimum of all
nonnegative integers $n$, such that there exists an exact sequence
$$0\rightarrow G_n\rightarrow\cdots\rightarrow G_0\rightarrow  M \rightarrow 0$$
in which all the $G_i$ have $G$-dimension zero.
By \cite[Theorem 4.13]{AB}, if $M$ has finite Gorenstein dimension then $\gd_R(M)=\depth R-\depth_R(M)$.
By \cite[Theorem 1.4]{AGP},
$\gd_R(M)$ is bounded above by the complete intersection dimension, $\cid_R(M)$, of $M$ and if $\cid_R(M)<\infty$ then the equality holds.

The notion of the weak Gorenstein dimension was introduced in \cite{H}. An $R$--module $M$ is said to be of weak Gorenstein dimension zero,
written $\wgd_R(M)=0$, if $\Ext^i_R(M,R)=0$ for all $i>0$. If for some integer $t\geq1$ we have $\Ext^t_R(M,R)\neq0$ and $\Ext^i_R(M,R)=0$ for all $i>t$ then $\wgd_R(M)=t$. In all other cases, i.e. if $\Ext^i_R(M,R)\neq0$
for infinitely many integer $i>0$, then $\wgd_R(M)=\infty$.

Note that, by \cite[Theorem 4.13]{AB}, every module of finite Gorenstein dimension has finite weak Gorenstein dimension and $\gd_R(M)=\wgd_R(M)$. On the other hand, there are modules having finite weak
Gorenstein dimension but whose Gorenstein dimension is infinite (see \cite{JS}).

Let $P_1\overset{f}{\rightarrow}P_0\rightarrow M\rightarrow 0$ be a
finite projective presentation of $M$. The transpose of $M$, $\Tr M$,
is defined to be $\coker f^*$, where $(-)^* := \Hom_R(-,R)$, which satisfies in the exact sequence
\begin{equation} \label{n1}
0\rightarrow M^*\rightarrow P_0^*\rightarrow P_1^*\rightarrow \Tr M\rightarrow 0
\end{equation}
and is unique up to projective equivalence. Thus the minimal projective
presentations of $M$ represent isomorphic transposes of $M$.
Two modules $M$ and $N$ are called \emph{stably isomorphic} and write $\underline{M}\cong\underline{N}$ if $M\oplus P\cong N\oplus Q$ for some projective
modules $P$ and $Q$.

The composed functors $\mathcal{T}_k:=\Tr\Omega^{k-1}$ for $k>0$  introduced by Auslander and Bridger in \cite{AB}. If $\Ext^i_R(M,R)=0$ for some $i>0$, then it is easy to see that $\underline{\mathcal{T}_iM}\cong\underline{\Omega\mathcal{T}_{i+1}M}$.

We frequently use the following Theorem of Auslander and Bridger.
\begin{thm}\cite[Theorem 2.8]{AB}\label{a1}\emph{
Let $M$ be an $R$--module and $n\geq0$ an integer. Then there are exact sequences of functors:
$$0\rightarrow\Ext^1_R(\mathcal{T}_{n+1}M,-)\rightarrow\Tor_n^R(M,-)\rightarrow\Hom_R(\Ext^n_R(M,R),-)
\rightarrow\Ext^2_R(\mathcal{T}_{n+1}M,-),$$
$$\Tor_2^R(\mathcal{T}_{n+1}M,-)\rightarrow(\Ext^n_R(M,R)\otimes_R-)\rightarrow\Ext^n_R(M,-)\rightarrow
\Tor_1^R(\mathcal{T}_{n+1}M,-)\rightarrow0.$$
}
\end{thm}
\section{the depth formula}

Let $M$ and $N$ be $R$--modules. In the following, we investigate the connection between the vanishing of homology modules, $\Tor_{i>0}^R(M,N)$, and the vanishing of cohomology modules, $\Ext^{i>0}_R(\Tr M,N)$.
\begin{lem}\label{p3}
\emph{
Let $M$, $N$ be $R$--modules such that $M$ has reducible complexity. If $M$ is of $G$-dimension zero, and
$\Tor_i^R(M,N)=0$ for all $i>0$ then $\Ext^i_R(\Tr M,N)=0$ for all $i>0$.}
\end{lem}
\begin{proof}
Set $c=\cx_R(M)$, we argue by induction on $c$. If $c=0$ then $\pd_R(M)<\infty$ and so $\pd_R(M)=\gd_R(M)=0$. Therefore $\Tr M=0$ and we have nothing to prove.
As $\Tor_i^R(M,N)=0$ for all $i>0$, $\Ext^1_R(\mathcal{T}_{i+1}M,N)=0$ for all $i>0$ by Theorem \ref{a1}.
Since $\Ext^i_R(M,R)=0$ for all $i>0$ then $\underline{\mathcal{T}_iM}\cong\underline{\Omega\mathcal{T}_{i+1}M}$ for all $i>0$ and so
\begin{equation}\label{26}
\Ext^i_R(\mathcal{T}_tM,N)=0 \text{ for all } t>1 \text{ and } 1\leq i<t.
\end{equation}
Suppose that $c>0$ and that $\eta\in\Ext^*_R(M,M)$ reduces the complexity of $M$. Consider The exact sequence
\begin{equation}\label{9}
0\rightarrow M\rightarrow K_{\eta}\rightarrow\Omega^q_R(M)\rightarrow0,
\end{equation}
where $|\eta|=q+1$ and $\cx_R(K_{\eta})<c$. Note that $\gd_R(K_{\eta})=\gd_R(M)=0$.
The exact sequence (\ref{9}), induces the long exact sequence
$$\cdots\rightarrow\Tor_i^R(M,N)\rightarrow\Tor_i^R(K_{\eta},N)\rightarrow\Tor_{i+q}^R(M,N)\rightarrow\cdots,$$
of homology modules. Therefore, $\Tor_i^R(K_{\eta},N)=0$ for all $i>0$ and so by induction hypothesis $\Ext^i_R(\Tr K_{\eta},N)=0$ for all $i>0$.
By \cite[Lemma 3.9]{AB}, from the exact sequence (\ref{9}), we obtain the following exact sequence
$$ 0\rightarrow {(\Omega^q_R(M))}^*\rightarrow {K_{\eta}}^*\rightarrow M^*\rightarrow \T(\Omega^q_R(M))\rightarrow \T(K_{\eta})\rightarrow \T(M)\rightarrow0$$
where $\underline{\T(M)}\cong\underline{\Tr M}$, $\underline{\T(K_{\eta})}\cong\underline{\Tr K_{\eta}}$ and $\underline{\T(\Omega^q_R(M))}\cong\underline{\mathcal{T}_{q+1}M}$. Since $\Ext^1_R(\Omega^q_R(M),R)=0$, we get the exact sequence
\begin{equation}\label{27}
0\rightarrow \T(\Omega^q_R(M))\rightarrow \T(K_{\eta})\rightarrow \T(M)\rightarrow0.
\end{equation}
The exact sequence (\ref{27}), induces a long exact sequence
\begin{equation}\label{333}
\cdots\rightarrow\Ext^i_R(\T(M),N)\rightarrow\Ext^i_R(\T(K_{\eta}),N)\rightarrow\Ext^i_R(\T(\Omega^q_R(M)),N)\rightarrow\cdots
\end{equation}
of cohomology modules.
As $\Ext^i_R(\Tr K_{\eta},N)=0$ for all $i>0$, we obtain from the (\ref{333})
$\Ext^{i+1}_R(\Tr M,N)\cong\Ext^i_R(\mathcal{T}_{q+1}M,N)$  for all $i>0$ and since
$\underline{\mathcal{T}_iM}\cong\underline{\Omega\mathcal{T}_{i+1}M}$ for all $i>0$,
\begin{equation}\label{11}
\Ext^i_R(\mathcal{T}_{q+1}M,N)\cong\Ext^{i+1}_R(\Tr M,N)\cong\Ext^{i+q+1}_R(\mathcal{T}_{q+1}M,N)  \text{  for all } i>0.
\end{equation}
Therefore if $q>0$ then by (\ref{11}) and (\ref{26})
\begin{equation}\label{29}
 \Ext^i_R(\mathcal{T}_{q+1}M,N)=0  \text{  for  } i\neq j(q+1) \text{ and } j>0,
\end{equation}
\begin{equation}\label{30}
\Ext^{q+1}_R(\mathcal{T}_{q+1}M,N)\cong\Ext^{j(q+1)}_R(\mathcal{T}_{q+1}M,N) \text{   for all }  j>0.
\end{equation}
By \cite[Lemma 2.3]{B2}, there exists an exact sequence
\begin{equation}
0\rightarrow\Omega^{q+1}_R(K_{\eta})\rightarrow K_{\eta^2}\oplus F\rightarrow K_{\eta}\rightarrow0,
\end{equation}
where $F$ is free. As $\gd_R(K_{\eta})=0$, by \cite[Lemma 3.9]{AB} we obtain the following exact sequence
\begin{equation}\label{12}
0\rightarrow\T(K_{\eta})\rightarrow\T(K_{\eta^2} \oplus F)\rightarrow\T(\Omega^{q+1}_R(K_{\eta}))\rightarrow0,
\end{equation}
where
$\underline{\T(K_{\eta})}\cong\underline{\Tr K_{\eta}}$, $\underline{\T(K_{\eta^2}\oplus F)}\cong\underline{\Tr K_{\eta^2}}$ and
$\underline{\T(\Omega^{q+1}_R(K_{\eta}))}\cong\underline{\mathcal{T}_{q+2}K_{\eta}}$.
The exact sequence (\ref{12}), induces a long exact sequence
\begin{equation}\label{13}
\cdots\rightarrow\Ext^i_R(\T(\Omega^{q+1}_R(K_{\eta})),N)\rightarrow\Ext^i_R(\T(K_{\eta^2}\oplus F),N)\rightarrow\Ext^i_R(\T(K_{\eta}),N)\rightarrow\cdots
\end{equation}
of cohomology modules. Note that by the proof of \cite[Proposition 2.2(ii)]{B2}, $\Omega^{q+1}_R(K_{\eta})$ has also reducible complexity. Therefore, by induction hypothesis, $\Ext^i_R(\mathcal{T}_{q+2}K_{\eta},N)=0$ for all $i>0$ and so by (\ref{13}), $\Ext^i_R(\Tr K_{\eta^2},N)=0$
for all $i>0$.
By \cite[Lemma 3.9]{AB}, from the exact sequence $0\rightarrow M\rightarrow K_{\eta^2}\rightarrow\Omega^{2q+1}_R(M)\rightarrow0$, we obtain the following exact sequence
\begin{equation}\label{14}
0\rightarrow{(\Omega^{2q+1}_R(M))}^*\rightarrow {(K_{\eta^2})}^*\rightarrow M^*\rightarrow\T(\Omega^{2q+1}_R(M))\rightarrow\T(K_{\eta^2})\rightarrow\T(M)\rightarrow0,
\end{equation}
where $\underline{\T(M)}\cong\underline{\Tr M}$, $\underline{\T(\Omega^{2q+1}_R(M))}\cong\underline{\mathcal{T}_{2q+2}M}$ and $\underline{\T(K_{\eta^2})}\cong\underline{\Tr K_{\eta^2}}$. As $\Ext^{2q+2}_R(M,R)=0$, we get the exact sequence
\begin{equation}\label{15}
0\rightarrow\T(\Omega^{2q+1}_R(M))\rightarrow\T(K_{\eta^2})\rightarrow\T(M)\rightarrow0.
\end{equation}
The exact sequence (\ref{15}), induces a long exact sequence
\begin{equation}\label{16}
\cdots\rightarrow\Ext^i_R(\T(M),N)\rightarrow\Ext^i_R(\T(K_{\eta^2}),N)\rightarrow\Ext^i_R(\T(\Omega^{2q+1}_R(M)),N)\rightarrow\cdots
\end{equation}
of cohomology modules.
As $\Ext^i_R(\Tr K_{\eta^2},N)=0$ and $\underline{\mathcal{T}_iM}\cong\underline{\Omega\mathcal{T}_{i+1}M}$ for all $i>0$, by (\ref{16}) we get the following isomorphisms.
\begin{equation}\label{28}
\Ext^i_R(\mathcal{T}_{2q+2}M,N)\cong\Ext^{i+1}_R(\Tr M,N)\cong\Ext^{2q+i+2}_R(\mathcal{T}_{2q+2}M,N)  \text{  for all } i>0.
\end{equation}
If $q=0$ then by (\ref{28}), (\ref{11}) and (\ref{26}), it is obvious that $\Ext^i_R(\Tr M,N)=0$ for all $i>0$. Now if $q>0$ then
by (\ref{28}) and (\ref{26}), $$\Ext^{2q+2}_R(\mathcal{T}_{q+1}M,N)\cong\Ext^{3q+3}_R(\mathcal{T}_{2q+2}M,N)\cong\Ext^{q+1}_R(\mathcal{T}_{2q+2}M,N)=0$$
Therefore by (\ref{29}) and (\ref{30}), $\Ext^i_R(\mathcal{T}_{q+1}M,N)=0$ for all $i>0$ and so $\Ext^i_R(\Tr M,N)=0$ for all $i>0$.
\end{proof}
The following Theorem is a generalization of \cite[Theorem 2.5]{AY}, \cite[Corollary 3.4]{BJ} and also \cite[Theorem 3.1]{BJ}.
\begin{thm}\emph{
Let $M$ and $N$ be $R$--modules and let $\Tor_i^R(M,N)=0$ for $i\gg0$. If $M$ has reducible complexity and $q=\sup\{i\mid\Tor_i^R(M,N)\neq0\}$ then the following statements hold true.
\begin{itemize}
      \item[(i)] If $\gd_R(M)<\infty$ and $q=0$ then $$\depth_R(M)+\depth_R(N)=\depth_R(M\otimes_RN)+\depth R.$$
        \item[(ii)] If $q>0$, $\depth_R(\Tor_q^R(M,N))\leq1$ then
                       $$\depth_R(M)+\depth_R(N)=\depth R+\depth_R(\Tor_q^R(M,N))-q$$
\end{itemize}
}
\end{thm}
\begin{proof}
(i) We argue by induction on $c=\cx_R(M)$. If $c=0$ then $\pd_R(M)<\infty$ and the formula holds by Auslander's original result, so suppose that $c>0$ and that $\eta\in\Ext^*_R(M,M)$ reduces the complexity of $M$.
The exact sequence $0\rightarrow M\rightarrow K_{\eta}\rightarrow\Omega^n_R(M)\rightarrow0$, induces a long exact sequence
\begin{equation}\label{19}
\cdots\rightarrow\Tor_i^R(M,N)\rightarrow\Tor_i^R(K_{\eta},N)\rightarrow\Tor_{i+n}^R(M,N)\rightarrow\cdots
\end{equation}
of homology modules. Therefore $\Tor_i^R(K_{\eta},N)=0$ for all $i>0$. As $\cx_R(K_{\eta})<c$ and $\gd_R(K_{\eta})<\infty$,
\begin{equation}\label{20}
\depth_R(K_{\eta})+\depth_R(N)=\depth_R(K_{\eta}\otimes_RN)+\depth R
\end{equation}
by induction hypothesis.
Now by induction on $\gd_R(M)$, we show that the formula holds. If $\gd_R(M)=0$ then by the Lemma \ref{p3}, $\Ext^i_R(\Tr M,N)=0$ for all $i>0$.
Hence by Theorem \ref{a1}, $M\otimes_RN\cong\Hom_R(M^*,N)$ and also by the exact sequence (\ref{n1}), $\Ext^i_R(M^*,N)=0$ for all $i>0$.
Therefore by \cite[Lemma 4.1]{AY}, $\depth_R(M\otimes_RN)=\depth_R(\Hom_R(M^*,N))=\depth_R(N)$ and so by the Auslander-Bridger formula,
$\depth_R(M)+\depth_R(N)=\depth_R(M\otimes N)+\depth R$.

Now let $\gd_R(M)>0$, if $n=0$ then we obtain the following exact sequence
\begin{equation}\label{21}
0\rightarrow M\otimes_RN\rightarrow K_{\eta}\otimes_RN\rightarrow M\otimes_RN\rightarrow0.
\end{equation}
As $\gd_R(M)>0$, $\gd_R(\Omega_R(M))=\gd_R(M)-1$. Note that by the proof of \cite[Proposition 2.2(ii)]{B2}, $\Omega_R(M)$ has also reducible complexity. Therefore,
\begin{equation}\label{22}
\depth_R(\Omega_R(M))+\depth_R(N)=\depth_R(\Omega_R(M)\otimes N)+\depth R
\end{equation}
by induction hypothesis.
From the exact sequence $0\rightarrow\Omega_R(M)\rightarrow F\rightarrow M\rightarrow0$, where $F$ is a free module, we obtain the exact sequence
$0\rightarrow\Omega_R(M)\otimes_RN\rightarrow F\otimes_RN\rightarrow M\otimes_RN\rightarrow0$. Therefore, $\depth_R(M\otimes_RN)\geq\min\{\depth_R(N),\depth_R(\Omega_R(M)\otimes_R N)-1\}$, by the depth Lemma. Now by (\ref{22}), $\depth_R(\Omega_R( M)\otimes_RN)-1=\depth_R(N)-\gd_R(M)<\depth_R(N)$ and so $\depth_R(M\otimes_RN)\geq\depth_R(N)-\gd_R(M)$. On the other hand, if $\depth_R(M\otimes_RN)>\depth_R(N)-\gd_R(M)$ then by the exact sequence (\ref{21}), it is obvious that $\depth_R(K_{\eta}\otimes_RN)>\depth_R(N)-\gd_R(M)=\depth_R(N)-\gd_R(K_{\eta})$, which is a contradiction by (\ref{20}). Hence
$\depth_R(M)+\depth_R(N)=\depth_R(M\otimes N)+\depth R$.

Now let $n>0$, then $\gd_R(\Omega^n_R(M))=\max\{0,\gd_R(M)-n\}<\gd_R(M)$. Note that by the proof of the \cite[Proposition 2.2(ii)]{B2}, $\Omega^n_R(M)$ has also reducible complexity and so by induction hypothesis, $\depth_R(\Omega^n_R(M)\otimes_RN)=\depth_R(N)-\gd_R(\Omega^n_R(M))$. Therefore, as $\gd_R(M)=\gd_R(K_{\eta})$, $\depth_R(K_{\eta}\otimes_RN)<\depth_R(\Omega^n_R(M)\otimes_RN)$ by (\ref{20}) and so from the exact sequence $0\rightarrow M\otimes_RN\rightarrow K_{\eta}\otimes_RN\rightarrow \Omega^n_R(M)\otimes_RN\rightarrow0$, it is obvious that $\depth_R(M\otimes_RN)=\depth_R(K_{\eta}\otimes_RN)$.
Therefore by (\ref{20}), $\depth_R(M)+\depth_R(N)=\depth_R(M\otimes N)+\depth R$.

(ii) We argue by induction on $\cx_R(M)=c$. If $c=0$ then $\pd_R(M)<\infty$ and the formula holds by Auslander's original result, so suppose that $c>0$ and that $\eta\in\Ext^*_R(M,M)$ reduces the complexity of $M$. The exact sequence $0\rightarrow M\rightarrow K_{\eta}\rightarrow\Omega^n_R(M)\rightarrow0$, induces a long exact sequence
\begin{equation}\label{23}
\cdots\rightarrow\Tor_q^R(M,N)\overset{f}\rightarrow\Tor_q^R(K_{\eta},N)\rightarrow\Tor_{q+n}^R(M,N)\rightarrow\cdots
\end{equation}
of homology modules.
From (\ref{23}), it is obvious that $q=\sup\{i\mid\Tor_i^R(K_{\eta},N)\neq0\}$. If $n=0$ then from (\ref{23}), we obtain the following exact sequences
\begin{equation}\label{24}
0\rightarrow\Tor_q^R(M,N)\rightarrow\Tor_q^R(K_{\eta},N)\rightarrow\coker(f)\rightarrow0,
\end{equation}
\begin{equation}\label{25}
0\rightarrow\coker(f)\rightarrow\Tor_q^R(M,N).
\end{equation}
If $\depth_R(\Tor_q^R(M,N))=0$, then by the exact sequence (\ref{24}), $\depth_R(\Tor_q^R(K_{\eta},N))=0$. As $\cx_R(K_{\eta})<c$ , by induction hypothesis, $\depth_R(N)+\depth_R(K_{\eta})=\depth R-q$ and since $\depth_R(K_{\eta})=\depth_R(M)$, we are done. If $\depth_R(\Tor_q^R(M,N))=1$ then by the exact sequence (\ref{25}), $\depth_R(\coker(f))>0$ and so by the exact sequence (\ref{24}), $\depth_R(\Tor_q^R(K_{\eta},N))=1$. Therefore by induction hypothesis, $\depth R+1-q=\depth_R(K_{\eta})+\depth_R(N)=\depth_R(M)+\depth_R(N)$.

Now suppose that $n>0$, then from the exact sequence (\ref{23}), it is obvious that $\Tor_q^R(M,N)\cong\Tor_q^R(K_{\eta},N)$ and so by induction hypothesis, we are done.
\end{proof}
The following lemma is useful for the rest of the paper.
\begin{lem}\label{l1}
\emph{
For an $R$--module $M$, $\cid_R(M)=0$ if and only if $\cid_R(\Tr M)=0$.}
\end{lem}
\begin{proof}
If $\cid_R(M)=0$ then $\cid_R(M^*)=0$ by \cite[Lemma 3.5]{BJ1} and so from the exact sequence (\ref{n1}), $\cid_R(\Tr M)<\infty$.
As $\gd_R(M)=0$, $\gd_R(\Tr M)=0$ by \cite[Lemma 4.1]{AB} and so $\cid_R(\Tr M)=0$ by \cite[Theorem 1.4]{AGP}.
As $\underline{M}\cong\underline{\Tr\Tr M}$, the other side is obvious.

\end{proof}
Let $M$ and $N$ be $R$--modules. In the following, we investigate the connection between complete intersection dimension of $M$ and the vanishing of cohomology modules, $\Ext^{i>0}_R(\Tr M,N)$, and the vanishing of homology modules, $\Tor_{i>0}^R(M,N)$.
\begin{prop}\emph{
Let $M$ and $N$ be $R$--modules such that $\cid_R(M)<\infty$. If two of the following conditions hold true then the third one is also true.}
\begin{itemize}
       \item[(i)]\emph{$\Tor_i^R(M,N)=0$ for all $i>0$},
       \item[(ii)]\emph{$\Ext^i_R(\Tr M,N)=0$ for all $i>0$},
       \item[(iii)]\emph{$\cid_R(M)=0$}.
\end{itemize}
\end{prop}
\begin{proof}
(i),(ii)$\Rightarrow$(iii) By \cite[Theorem 2.5]{AY} and \cite[Theorem 1.4]{AGP}, $\depth_R(M\otimes_RN)=\depth_R(N)-\cid_R(M)$. As $\Ext^i_R(\Tr M,N)=0$ for all $i>0$, $M\otimes_RN\cong\Hom_R(M^*,N)$ by Theorem \ref{a1}
and also $\Ext^i_R(M^*,N)=0$ for all $i>0$, by the exact sequence (\ref{n1}). Therefore, $\depth_R(\Hom_R(M^*,N))=\depth_R(N)$ by \cite[Lemma 4.1]{AY}. Hence $\cid_R(M)=0$.

(ii),(iii)$\Rightarrow$(i) Set $K=\Tr M$ and $c=\cx_R(M)$. By Lemma \ref{l1}, $\cid_R(K)=0$. By Theorem \ref{a1}, $\Tor_1^R(\mathcal{T}_{i+1}K,N)=0$ for all $i>0$. As $\Ext^i_R(K,R)=0$ for all $i>0$, $\underline{\mathcal{T}_iK}\cong\underline{\Omega\mathcal{T}_{i+1}K}$ for all $i>0$ and so
$\Tor_i^R(\mathcal{T}_jK,N)=0$ for all $j>1$ and $1\leq i<j$. As $\underline{M}\cong\underline{\Tr K}$, $\cx_R(\Tr K)=c$ and so $\cx_R(\mathcal{T}_iK)=c$ for all $i>0$. Since $\Tor_i^R(\mathcal{T}_{c+2}K,N)=0$ for $1\leq i\leq c+1$ and $\cid_R(\mathcal{T}_{c+2}K)=0$, then $\Tor_i^R(\mathcal{T}_{c+2}K,N)=0$ for all $i>0$, by\cite[Corollary 2,6]{J1}. Therefore, $\Tor_i^R(\mathcal{T}_1K,N)=0$ for all $i>0$. As $\underline{M}\cong\underline{\Tr K}$, $\Tor^i_R(M,N)=0$ for all $i>0$.

(i), (iii)$\Rightarrow$(ii) By \cite[Theorem 1.4]{AGP}$, \gd_R(M)=\cid_R(M)=0$ and by \cite[Proposition 2.2(i)]{B2}, $M$ has reducible complexity. Therefore, $\Ext^i_R(\Tr M,N)=0$ for all $i>0$ by Lemma \ref{p3}.
\end{proof}
\section{ Vanishing results}

Let $M$ and $N$ be $R$--modules. In \cite{J1}, Jorgensen proved that the vanishing of Ext for a certain
sequence of numbers forces the vanishing of all the higher Ext groups. More precisely, he proved that if $\cid_R(M)<\infty$ and
$\Ext^i_R(M,N)=\Ext^{i+1}_R(M,N)=\cdots=\Ext^{i+c}_R(M,N)=0$, where $c=\cx_R(M)$, then $\Ext^j_R(M,N)=0$ for all $j>\cid_R(M)$
\cite[Corollary 2.6]{J1}. In \cite{B1}, Bergh assumed the vanishing of nonconsecutive Ext groups and generalized this result.
For an $R$--module $M$ of finite complete intersection dimension, he proved that if there exist an odd number $q\geq1$, and a number $n>\cid_R(M)$ such that $\Ext^i_R(M,N)=0$ for $i\in\{n,n+q,\cdots,n+cq\}$, then $\Ext^i_R(M,N)=0$ for all $i>\cid_R(M)$, \cite[Theorem 3.1]{B1}. In this section, we are going to prove similar results, when $\wgd_R(M)<\infty$ and $N$ has finite complete intersection dimension.

\begin{prop}\label{p1}
\emph{
Let $M$ and $N$ be $R$--modules and let $\cid_R(N)<\infty$  and $\wgd_R(M)<\infty$. Set $\cx_R(N)=c$. If $\Ext^i_R(M,N)=\Ext^{i+1}_R(M,N)=\ldots=\Ext^{i+c}_R(M,N)=0$, for some
$i>\wgd_R(M)$, then $\Ext^j_R(M,N)=0$ for all $j>\wgd_R(M)$.}
\end{prop}
\begin{proof}
We argue by induction on $c$. If $c=0$ then $\pd_R(N)<\infty$.
As $\Ext^i_R(M,R)=0$ for all $i>\wgd_R(M)$, $\underline{\mathcal{T}_iM}\cong\underline{\Omega\mathcal{T}_{i+1}M}$ for all $i>\wgd_R(M)$ and also
by Theorem \ref{a1}, $\Ext^j_R(M,N)\cong\Tor_1^R(\mathcal{T}_{j+1}M,N)$ for all $j>\wgd_R(M)$. Set $t=\pd_R(N)$. Therefore
$\Ext^j_R(M,N)\cong\Tor_{t+1}^R(\mathcal{T}_{t+j+1}M,N)=0$ for all $j>\wgd_R(M)$.

Now suppose $c$ is positive and set $n=\wgd_R(M)$. As $\cid_R(N)<\infty$, by \cite[Lemma 2.1(i)]{B1}, there exists a quasi deformation $R\rightarrow A\twoheadleftarrow Q$ such that the $A$--module $A\otimes_RN$ has reducible complexity by an element $\eta\in\Ext^2_A(A\otimes_RN,A\otimes_RN)$. Set $\hat{N}=A\otimes_RN$ and $\hat{M}=A\otimes_RM$. Note that $\wgd_A(\hat{M})=n$ and by the proof of \cite[Lemma 2.1(i)]{B1}, $\cid_A(\hat{N})<\infty$. The exact sequence $$0\rightarrow\hat{N}\rightarrow K_{\eta}\rightarrow \Omega_{A}\hat{N}\rightarrow0$$ induces a long exact sequence
\begin{equation}\label{1}
\cdots\rightarrow\Ext^j_A(\hat{M},\hat{N})
\rightarrow\Ext^j_A(\hat{M},K_{\eta})\rightarrow
\Ext^j_A(\hat{M},\Omega_A(\hat{N})){\rightarrow}\cdots
\end{equation}
of cohomology modules.
Now consider the exact sequence $0\rightarrow\Omega_A(\hat{N})\rightarrow F\rightarrow\hat{N}\rightarrow0$, where $F$ is a free $A$--module.
As $\Ext^k_A(\hat{M},A)=0$ for all $k>n$, we get the following isomorphism
\begin{equation}\label{2}
\Ext^{j-1}_A(\hat{M},\hat{N})\cong\Ext^j_A(\hat{M},\Omega_A(\hat{N}))
\end{equation}
for all $j>n+1$. Now from (\ref{1}) and (\ref{2}), it is obvious that $\Ext^j_A(\hat{M},K_{\eta})=0$ for $i+1\leq j\leq i+c$.
As $\cx_A(K_{\eta})< \cx_A(\hat{N})=c$, by induction hypothesis we have $\Ext^j_A(\hat{M},K_{\eta})=0$ for all $j>n$.
Therefore from (\ref{1}) and (\ref{2}), we get $$\Ext^{j-1}_A(\hat{M},\hat{N})\cong\Ext^j_A(\hat{M},\Omega_A(\hat{N}))\cong\Ext^{j+1}_A(\hat{M},\hat{N})$$
for all $j>n+1$. Now since $c>0$, it is obvious that $\Ext^j_A(\hat{M},\hat{N})=0$ for all $j>n$. Therefore $\Ext^j_R(M,N)=0$ for all $j>n$.
\end{proof}
Now we can generalize Proposition \ref{p1} as follows.
\begin{thm}\label{t1}
\emph{
Let $M$ and $N$ be $R$--modules and let $\cid_R(N)<\infty$ and $\wgd_R(M)<\infty$. Set $\cx_R(N)=c$. If there exist an odd number $q\geq1$, and $i>\wgd_R(M)$ such that $\Ext^j_R(M,N)=0$, for $j\in\{i,i+q,\cdots,i+cq\}$, then
$\Ext^j_R(M,N)=0$ for all $j>\wgd_R(M)$.}
\end{thm}
\begin{proof}
We argue by induction on $c$. If $c=0$ then $\pd_R(N)<\infty$ and as we have seen in the proof of Proposition \ref{p1}, $\Ext^j_R(M,N)=0$ for all $j>\wgd_R(M)$.

Now let $c>0$, $q=2t-1$, $t\geq1$. As $\cid_R(N)<\infty$, by \cite[Lemma 2.1(i)]{B1}, there exists a quasi deformation $R\rightarrow A\twoheadleftarrow Q$ such that the $A$--module $A\otimes_RN$ has reducible complexity by an element $\eta\in\Ext^2_A(A\otimes_RN,A\otimes_RN)$. Set $\hat{N}=A\otimes_RN$ and $\hat{M}=A\otimes_RM$. As $\cid_A(\hat{N})<\infty$, by \cite[Lemma 2.1(ii)]{B1}, the element $\eta^t$ also reduces the complexity of $\hat{N}$. From the exact sequence
$$0\rightarrow\hat{N}\rightarrow K_{\eta^t}\rightarrow\Omega^q_A(\hat{N})\rightarrow0$$
we get the following long exact sequence of cohomology modules.
\begin{equation}\label{3}
\cdots\rightarrow\Ext^j_A(\hat{M},\hat{N})\rightarrow\Ext^j_A(\hat{M},K_{\eta^t})\rightarrow\Ext^j_A(\hat{M},\Omega^q_A(\hat{N}))
\rightarrow\cdots.
\end{equation}
As $\wgd_A(\hat{M})=\wgd_R(M)<\infty$, it is easy to see that
\begin{equation}\label{4}
\Ext^j_A(\hat{M},\Omega^q_A(\hat{N}))\cong\Ext^{j-q}_A(\hat{M},\hat{N})
\end{equation}
for all
$j>\wgd_A(\hat{M})+q$. Now from (\ref{3}) and (\ref{4}), it is obvious that $\Ext^j_A(\hat{M},K_{\eta^t})=0$ for $j\in\{i+q,i+2q,\cdots,i+cq\}$. As $\cx_A(K_{\eta^t})<\cx_A(\hat{N})=c$, by induction hypothesis we have
$\Ext^j_A(\hat{M},K_{\eta^t})=0$ for all $j>\wgd_A(\hat{M})$. Therefore from (\ref{3}) and (\ref{4}), we get
$$\Ext^j_A(\hat{M},\hat{N})\cong\Ext^{j-1}_A(\hat{M},\Omega^q_A(\hat{N}))\cong\Ext^{j-1-q}_A(\hat{M},\hat{N})$$ for all
$j>\wgd_A(\hat{M})+q+1$. Now it is easy to see that, $\Ext^j_A(\hat{M},\hat{N})=0$ for $i+cq\leq j\leq i+cq+c$.
Therefore, by Proposition \ref{p1}, $\Ext^j_A(\hat{M},\hat{N})=0$ for all $j>\wgd_A(\hat{M})$, and so
$\Ext^j_R(M,N)=0$ for all $j>\wgd_R(M)$.
\end{proof}
\begin{cor}\label{c}\emph{
Let $M$ and $N$ be $R$--modules of finite complete intersection dimensions. Set $c=\min\{\cx_R(M),\cx_R(N)\}$,
if there exist an odd number $q\geq1$, and $i>\cid_R(M)$ such that $\Ext^j_R(M,N)=0$, for $j\in\{i,i+q,\cdots,i+cq\}$, then
$\Ext^j_R(M,N)=0$ for all $j>\cid_R(M)$.}
\end{cor}
\begin{proof}
 By \cite[Theorem 1.4]{AGP}, $\wgd_R(M)=\gd_R(M)=\cid_R(M)$. Now the assertion is obvious by \cite[Theorem 3.1]{B1} and Theorem \ref{t1}.
\end{proof}

Let $M$ and $N$ be $R$--modules. It is well-known that if $M$ has finite Gorenstein dimension and $\pd_R(N)<\infty$ then
$\gd_R(M)=\sup\{i\mid\Ext^i_R(M,N)\neq0\}$. In the following, we generalize this result for modules with reducible complexity.
\begin{thm}\label{t2}\emph{
Let $M$, $N$ be nonzero $R$--modules. If $N$ has reducible complexity, $\wgd_R(M)<\infty$ and $\Ext^i_R(M,N)=0$ for $i\gg0$ then
\begin{itemize}
     \item[(i)] $\Ext^{\tiny{\wgd_R(M)}}_R(M,N)\cong\Ext^{\tiny{\wgd_R(M)}}_R(M,R)\otimes_R N$,
     \item[(ii)] $\wgd_R(M)=\sup\{i\mid\Ext^i_R(M,N)\neq0\}$.
\end{itemize}
}
\end{thm}
\begin{proof}
Set $n=\wgd_R(M)$ and $c=\cx_R(N)$. First we show that $\Ext^i_R(M,N)=0$ for all $i>n$. We argue by induction on $c$. If $c=0$ then $\pd_R(N)<\infty$ and so as we have seen in the proof of Proposition \ref{p1}, $\Ext^i_R(M,N)=0$ for all $i>n$. Now let $c>0$ and $\eta\in\Ext^*_R(N,N)$ reduces the complexity of $N$. Consider the exact sequence
\begin{equation}\label{31}
0\rightarrow N\rightarrow K_{\eta}\rightarrow\Omega^q_R(N)\rightarrow0,
\end{equation}
 where $q=|\eta|-1$ and $\cx_R(K_{\eta})<c$. The exact sequence (\ref{31}), induces a long exact sequence
\begin{equation}\label{5}
\cdots\rightarrow\Ext^i_R(M,N)\rightarrow\Ext^i_R(M,K_{\eta})\rightarrow\Ext^i_R(M,\Omega^q_R(N))\rightarrow\cdots
\end{equation}
of cohomology modules. As $\Ext^i_R(M,R)=0$ for $i>n$, it is easy to see that
\begin{equation}\label{6}
\Ext^i_R(M,\Omega^q_R(N))\cong\Ext^{i-q}_R(M,N)
\end{equation}
for $i>n+q$. As $\Ext^i_R(M,N)=0$ for $i\gg0$, from (\ref{5}) and (\ref{6}) we see that $\Ext^i_R(M,K_{\eta})=0$ for $i\gg0$ and so by induction hypothesis $\Ext^i_R(M,K_{\eta})=0$ for all $i>n$. Therefore by (\ref{5}) and (\ref{6}), we have
$$\Ext^{i-q}_R(M,N)\cong\Ext^i_R(M,\Omega^q_R(N))\cong\Ext^{i+1}_R(M,N)$$ for $i\geq n+q+1$ and so $\Ext^i_R(M,N)=0$ for all $i>n$.

Now we show that if $\Ext^i_R(M,N)=0$ for all $i>n$ then $\Tor_i^R(\mathcal{T}_{n+1}M,N)=0$ for all $i>0$. We argue by induction on $c$. If $c=0$ then $\pd_R(N)<\infty$. As $\Ext^i_R(M,R)=0$ for all
$i>n$, $\underline{\mathcal{T}_iM}\cong\underline{\Omega\mathcal{T}_{i+1}M}$ for all $i>n$ and so it is obvious that $\Tor_i^R(\mathcal{T}_{n+1}M,N)=0$ for all $i>0$. Now let $c>0$ and $\eta\in\Ext^*_R(N,N)$ reduces the complexity of $N$. Consider the exact sequence
\begin{equation}\label{7}
0\rightarrow N\rightarrow K_{\eta}\rightarrow\Omega^q_R(N)\rightarrow0,
\end{equation}
where $\cx_R(K_{\eta})<c$. As we have seen in the proof of the first part $\Ext^i_R(M,K_{\eta})=0$ for all $i>n$ and so by induction hypothesis $\Tor_i^R(\mathcal{T}_{n+1}M,K_{\eta})=0$ for all $i>0$. As $\Ext^i_R(M,K_{\eta})=0$ for all $i>n$, by Theorem \ref{a1}, $\Tor_1^R(\mathcal{T}_{i+1}M,K_{\eta})=0$ for all $i>n$ and since $\underline{\mathcal{T}_iM}\cong\underline{\Omega\mathcal{T}_{i+1}M}$ for all $i>n$, it is easy to see that $\Tor_i^R(\mathcal{T}_{j}M,K_{\eta})=0$ for all $i>0$ and $j\geq n+1$.
Set $t=q+n+2$. The exact sequence (\ref{7}) induces the long exact sequence
\begin{equation}
\cdots\rightarrow\Tor_i^R(\mathcal{T}_tM,N)\rightarrow\Tor_i^R(\mathcal{T}_tM,K_{\eta})\rightarrow\Tor_{i+q}^R(\mathcal{T}_tM,N)
\rightarrow\cdots
\end{equation}
of homology modules. Therefore
\begin{equation}\label{8}
\Tor_i^R(\mathcal{T}_tM,N)\cong\Tor_{i+q+1}^R(\mathcal{T}_tM,N) \text{   for all  } i>0.
\end{equation}
As $\Ext^i_R(M,N)=0$ for all $i>n$, $\Tor_1^R(\mathcal{T}_iM,N)=0$ for all $i>n+1$ by Theorem \ref{a1}. As $\underline{\mathcal{T}_iM}\cong\underline{\Omega\mathcal{T}_{i+1}M}$ for all $i>n$,
$\Tor_i^R(\mathcal{T}_tM,N)\cong\Tor_1^R(\mathcal{T}_{t-i+1}M,N)=0$ for $1\leq i\leq q+1$ and so $\Tor_i^R(\mathcal{T}_tM,N)=0$ for all $i>0$, by (\ref{8}). Therefore $\Tor_i^R(\mathcal{T}_{n+1}M,N)=0$ for $i>0$ and so $\Ext^n_R(M,R)\otimes_RN\cong\Ext^n_R(M,N)$, by Theorem \ref{a1}.
As $R$ is local and $N$, $\Ext^n_R(M,R)$ are non-zero, $\Ext^n_R(M,N)\neq0$ and so $n=\sup\{i\mid\Ext^i_R(M,N)\neq0\}$.
\end{proof}

\bibliographystyle{amsplain}
{\bf Acknowledgements.}
I am very grateful to Petter Andreas Bergh and my thesis adviser Mohammad Taghi Dibaei for their assistance in the preparation of this article. This work was done while the author was visiting Trondheim, Norway, Autumn 2011. I thank the Algebra Group at the Institutt for Matematiske Fag, NTNU, for their hospitality.

\end{document}